\documentclass[3p]{elsarticle}
 \usepackage{graphics}
 \usepackage{graphicx}
 \usepackage{epsfig}
\usepackage{amssymb}
 \usepackage{amsthm}
 \usepackage{lineno}
 \usepackage{amsmath}
   \numberwithin{equation}{section}
\usepackage{mathrsfs}
\usepackage{color}

\theoremstyle{plain}
\newtheorem{theorem}{Theorem}

\newtheorem{lemma}[theorem]{Lemma}

\newtheorem{definition}[theorem]{Definition}

\numberwithin{equation}{section} \numberwithin{theorem}{section}
\newproof{pot}{{\bf{Proof of Theorem \ref{Thm1.1}}\rm}}
\newenvironment{remark}[1][Remark:]{\begin{trivlist}
\item[\hskip \labelsep {\bfseries #1}]}{\end{trivlist}}

\begin{document}
\begin{frontmatter}
\author[a]{Yuxuan Zhang}
\ead{zhangyx595@nenu.edu.cn}
\author[a]{Xiaojun Chang\footnotemark[1]}
\ead{changxj100@nenu.edu.cn}
\author[b]{Lin Chen}
\ead{clzj008@163.com}

\address[a]{School of Mathematics and Statistics $\&$ Center for Mathematics and Interdisciplinary Sciences,\\
 Northeast Normal University, Changchun 130024, Jilin,
 PR China}
\address[b]{School of Mathematics and Statistics, Yili Normal University, Yining, 835000, China}

\footnotetext{\footnotemark[1]{Corresponding author.}}

\title{Standing waves with prescribed mass for NLS equations with Hardy potential in the half-space under Neumman boundary condition}

\begin{abstract}
 Consider the Neumann problem:
\begin{eqnarray*}
\begin{cases}
	&-\Delta u-\frac{\mu}{|x|^2}u +\lambda u =|u|^{q-2}u+|u|^{p-2}u ~~~\mbox{in}~~\mathbb{R}_+^N,~N\ge3,\\
    &\frac{\partial u}{\partial \nu}=0 ~~ \mbox{on}~~ \partial\mathbb{R}_+^N
\end{cases}
\end{eqnarray*}
with the prescribed mass:
    \begin{equation*}
         \int_{\mathbb{R}_+^N}|u|^2 dx=a>0,
    \end{equation*}
where $\mathbb{R}_+^N$ denotes the upper half-space in $\mathbb{R}^N$, $\frac{1}{|x|^2}$ is the Hardy potential, $2<q<2+\frac{4}{N}<p<2^*$, $\mu>0$, $\nu$ stands for the  outward unit normal vector to $\partial \mathbb{R}_+^N$, and $\lambda$ appears as a Lagrange  multiplier. Firstly, by applying Ekeland's variational principle, we establish the existence of normalized solutions that correspond to local minima of the associated energy functional. Furthermore, we find a second normalized solution of mountain pass type by employing a parameterized minimax principle that incorporates Morse index information. 
Our analysis relies on a Hardy inequality in $H^1(\mathbb{R}_+^N)$, as well as a Pohozaev identity involving the Hardy potential on $\mathbb{R}_+^N$.  This work provides a variational framework for investigating the existence of normalized solutions to the Hardy type system within a half-space, and our approach is flexible, allowing it to be adapted to handle more general nonlinearities.
\end{abstract}
\begin{keyword}
Normalized solutions; Nonlinear Schr\"{o}dinger equation; Half-space; Neumann boundary condition; Hardy potential.\\
{\noindent\text{\emph{Mathematics Subject Classification:}} 35A15, 35J20, 35J60, 35Q55.}
\end{keyword}

\end{frontmatter}

\section{Introduction}
\noindent

This paper is devoted to investigating the following  Neumann problem involving Hardy term:
\begin{eqnarray}\label{start}
\begin{cases}
	&-\Delta u-\frac{\mu}{|x|^2}u +\lambda u =|u|^{q-2}u+|u|^{p-2}u ~~~\mbox{in}~~\mathbb{R}_+^N,~N\ge3,\\
    &\frac{\partial u}{\partial \nu}=0 ~~ \mbox{on}~~ \partial\mathbb{R}_+^N
\end{cases}
\end{eqnarray}
with a mass constraint
\begin{equation}\label{masscon}
	    \int_{\mathbb{R}_+^N}|u|^2 dx=a,
\end{equation}
where $\mathbb{R}_+^N$ is a half-space, $\nu$ stands for the unit outer normal to $\partial \mathbb{R}_+^N$, $0<\mu<\bar{\mu}=\left(\frac{N-2}{2} \right)^2$, $2<q<2_*<p<2^*$, $2_*=2+\frac{4}{N}$, $2^*=\frac{2N}{N-2}$, $a>0$ is prescribed, $\lambda$ appears as a Lagrange multiplier.

The study of \eqref{start} is motivated by the pursuit of standing wave solutions to the following time-dependent nonlinear Schr\"{o}dinger equation (NLS):
\begin{equation}\label{tNLS}
	i\partial_t \psi +\Delta \psi + \frac{\mu}{|x|^2}\psi + g(|\psi|)\psi =0,~~(t, x)\in \mathbb{R}\times \Omega,
\end{equation}
where $\Omega$ is a domain in $\mathbb{R}^N$. In fact, by assuming $\psi(t, x) = e^{-i\lambda t}u(x)$ with $\lambda \in \mathbb{R}$ and  $u:\Omega\to \mathbb{R}$,
 the function $u(x)$ satisfies the stationary Schr\"{o}dinger equation
\begin{equation}\label{xianzai}
		-\Delta u+\lambda u -\frac{\mu}{|x|^2} u = f(u)~~\mbox{in}~~ \Omega,\\
\end{equation}
where $f(u)=g(|u|)u$.
 From a mathematical standpoint, the inverse square potential $\frac{\mu}{|x|^2}$ is of critical importance due to it possesses the same homogeneity as the Laplacian and does not belong to Kato's class. Therefore, it cannot be regarded as a lower-order perturbation term. Moreover, any nontrivial solution of Equ. \eqref{start} exhibits a singularity at $x=0$ when $\mu\not = 0$.
From a physical perspective, the nonlinear Schr\"{o}dinger equation with Hardy potential emerges in various physical contexts, such as nonrelativistic quantum mechanics, molecular physics, quantum cosmology, and the linearization of combustion models. 

When investigating solutions to \eqref{xianzai}, a common approach involves considering a fixed $\lambda\in \mathbb{R}$. In this context, many researchers have explored the existence of positive solutions, sign-changing solutions, and ground states using variational methods, subject to suitable assumptions on $\mu$ and $f$. When $\Omega$ is bounded, comprehensive studies are available in \cite{CaoYan, Ghoussoub-2017, Ghoussoub-2017-CVPDE, RuizWillem}. For the case where $\Omega=\mathbb{R}^N$, interested readers can refer to \cite{Felli, LLT2021, Smets, Terracini} and the references therein.

Another way to find the nontrivial solutions for \eqref{xianzai} is to search for solutions with prescribed mass, namely normalized solutions,
where $\lambda\in \mathbb{R}$ is unknown. 
This approach appears particularly significant due to the conservation of mass.
  When $\Omega=\mathbb{R}^N$ and $\mu=0$, the equation \eqref{xianzai}  with prescribed mass reduces to
\begin{eqnarray}\label{take}
\begin{cases}
	&	-\Delta u+\lambda u =f(u)~~\mbox{in}~~ \mathbb{R}^N,\\
    &\int_{\mathbb{R}^N}|u|^2 dx=a.
\end{cases}
\end{eqnarray}
Consider the energy functional $\Psi: H^1(\mathbb{R}^N)\to \mathbb{R}$ given by
\[\Psi(u)=\frac{1}{2}\int_{\mathbb{R}^N}|\nabla u|^2 dx-\int_{\mathbb{R}^N}F(u)dx, \]
where $F(u)=\int_0^u f(s) ds$. 
Set $\mathcal{S}_a=\left \{u\in H^1(\mathbb{R}^N):~\int_{\mathbb{R}^N}|u|^2=a \right \}$.
Clearly, normalized solutions of \eqref{take} correspond to the critical points of $\Psi|_{\mathcal{S}_a}$.
When the nonlinearity $f$ satisfies $L^2$-subcritical growth at infinity, meaning $f(u)$ behaves as $|u|^{p-2}u$ as $|u|\to+\infty$, where $p\in (2,2_*)$, Cazenave and Lions \cite{Cazenave-Lions} developed a concentration compactness argument to establish the existence of normalized ground state solutions of (\ref{take}), which serve as global minimizers of $\Psi$ on the constraint set $\mathcal{S}_a$. For further related results, please refer to \cite{Stuart,Shibata,Hirata-Tanaka}. If the range of $p$ falls within $(2_*,2^*)$, indicating that $f$ is $L^2$-supercritical at infinity, in this scenario, the functional $\Psi$ becomes unbounded from below on $\mathcal{S}_a$, rendering the minimization method on $\mathcal{S}_a$ ineffective. The pioneering work in this direction was carried out by Jeanjean \cite{Jeanjean1997}, who introduced a scaled functional $\tilde{\Psi}(u,t):=\Psi(t\star u)$ with $t\star u(\cdot)=t^{\frac{N}{2}}u(t\cdot)$.
It was proved that $\tilde{\Psi}$ satisfies the mountain pass geometry on $\mathcal{S}_a\times \mathbb{R}$.
Moreover, using the corresponding Pohozaev identity, he obtained a bounded (PS) sequence $\{u_n\}$ of $\Psi|_{\mathcal{S}_a}$ and thus a normalized solution of (\ref{take}).
In 2017, Bartsch and Soave \cite{Bartsch2017} introduced a minimax approach on the intersection $\mathcal{S}_a\cap \mathcal{M}$
 to investigate the $L^2$-supercritical NLS equations and systems on $\mathbb{R}^N$, where 
\[\mathcal{M}:=\left \{u\in H^1(\mathbb{R}^N):~\int_{\mathbb{R}^N}|\nabla u|^2 dx=\frac{N}{2}\int_{\mathbb{R}^N}\left(f(u)u-2F(u)\right) dx\right \}\]
 is the corresponding Pohozaev manifold.
The core of their proof lies in demonstrating that the manifold $\mathcal{M}$ is a natural constraint.
Later, building on this natural constraint argument, Bieganowski and Mederski \cite{Bieganowski} presented a constrained minimization method for problem \eqref{take}. They initially sought minimizers for $\Psi$ on $\mathcal{D}_a\cap \mathcal{M}$, where 
 \[\mathcal{D}_a:= \left \{u\in H^1(\mathbb{R}^N):~\int_{\mathbb{R}^N}u^2\leq a\right\},\]
and subsequently they showed that the minimizers stay on $\mathcal{S}_a\cap \mathcal{M}$.
For more related results, we refer the readers to \cite{IkNo,Jeanjean-Lu,AJM2022,CLY-2023} and the references therein.
 
In recent years, the study of normalized solutions for the nonlinear Schr\"odinger equation with combined power nonlinearities has attracted much attention.
 Soave \cite{Soave-2020JDE} employed the Pohozaev manifold decomposition to explore the existence, non-existence, multiplicity and orbital stability of normalized solutions for \eqref{take} with nonlinearity of the form $f(u)=\kappa |u|^{q-2}u+|u|^{p-2}u$, where $\kappa \in \mathbb{R}$ and $2<q\leq 2_* \leq p < 2^*$ ($p\neq q$).
 In particular, when $2<q< 2_* < p < 2^*$,
 Soave proved that \eqref{take} admits a normalized ground state solution, which serves as a local minimizer of $\Psi$, as well as a mountain pass type normalized solution. 
For further insights into the Sobolev critical Schr\"{o}dinger equation with combined nonlinearities, we refer readers to \cite{Soave-2020JFA, JJLV2022, JL2022, WW-2022}.

Significant efforts have also been made in studying normalized solutions for nonlinear Schr\"odinger equations with potentials. In 2021,
 Bartsch et al. \cite{BMRM2021} studied the existence of normalized solutions for the following equation:
\begin{eqnarray}\label{heihei}
\begin{cases}
	&	-\Delta u+ \lambda u +V(x)u=|u|^{p-2}u~~\mbox{in}~~ \mathbb{R}^N,\\
    &\int_{\mathbb{R}^N}|u|^2 dx=a,
\end{cases}
\end{eqnarray}
where $p\in (2_*,2^*)$.
They developed a new linking argument and obtained the existence of normalized solutions with high Morse index for positive vanishing potentials $V(x)\geq 0$. 
For the case of negative vanishing potentials $V(x)\leq 0$,
Molle et al. \cite{Molle} employed mountain pass arguments to establish the existence of normalized solutions for \eqref{heihei} under explicit smallness conditions on $V$. 
Subsequently, Liu and Zhao \cite{LZ-2024} obtained the existence of normalized ground state solutions to \eqref{heihei} using the constraint minimization method proposed in \cite{Bieganowski}, for the scenario where $\lim\limits_{|x|\to \infty}V(x)=V_\infty <+\infty$ including singular potentials  under general nonlinearities.
For more studies on \eqref{heihei}, see \cite{AW-2019,IM-2020, PPVV-2021,DZ-2022,BQZ-2024,FLT2024} and their references.

In a recent result \cite{LZ2023}, Li and Zou obtained the existence and properties of normalized solutions for the following equation involving Hardy potential:
\begin{eqnarray*}
\begin{cases}
	&	-\Delta u-\frac{\mu}{|x|^2}u+\lambda u= |u|^{2^*-2}u+\vartheta |u|^{p-2}u~~\mbox{in}~~ \mathbb{R}^N,\\
    &\int_{\mathbb{R}^N}|u|^2 dx=a.
\end{cases}
\end{eqnarray*}
For an $L^2$-subcritical, $L^2$-critical, or $L^2$-supercritical perturbation $\vartheta |u|^{p-2}u$, they established several existence results for normalized ground state when $\vartheta\geq 0$ and non-existence results when $\vartheta\leq 0$. Furthermore, they also considered the asymptotic behavior of normalized solutions as $\mu \to 0$ or $\vartheta\to 0$.

While the NLS equation has been extensively studied in the whole space scenario, much less is known about its behavior in half-space. In 1997,
 Cerami and Passaseo \cite{Cerami-half} studied the following problem concerning the NLS equation in $\mathbb{R}^N_+$:
 \begin{eqnarray}\label{lkzv}
\begin{cases}
	-\Delta u+V(x) u =u^{2^*-1} &\mbox{in} ~ \mathbb{R}_+^N,\\
	u>0 &\mbox{in} ~ \mathbb{R}_+^N,\\
    \frac{\partial u}{\partial \nu}=0 & \mbox{on}~\partial \mathbb{R}_+^N,\\
    u(x)\to 0 & \mbox{as}~|x|\to \infty.
\end{cases}
\end{eqnarray}
They showed that if $V$ is a non-negative function that is strictly positive somewhere, and its $\|V\|_{L^{N/2}(\mathbb{R}_+^N)}$ norm satisfies $\|V\|_{L^{N/2}(\mathbb{R}_+^N)}<S(1-2^{-\frac{2}{N}})$, where $S$ is the best Sobolev constant, then \eqref{lkzv} has a positive solution. Additionally, they provided conditions on
$V$ to ensure the existence of multiple positive solutions.
In their proof,
Cerami and Passaseo employed the reflection technique and the concentration compactness principle as key tools to address the compactness issue. Subsequently,
Alves \cite{Alves-2009}
studied the existence of positive solutions for a class of $p$-Laplacian Schr\"{o}dinger equations with Neumann boundary conditions in half-space. More recently, Fern\'andez and Weth \cite{Weth-2022} analyzed the existence, non-existence and multiplicity of bounded positive solutions to a Dirichlet boundary problem involving NLS equations in half-spaces.

Inspired by the aforementioned studies, a natural question arises: can we find normalized solutions for equations like \eqref{start},  which involves both the Hardy potential and combined nonlinearities, in a half-space?
Unlike the problem in the whole space $\mathbb{R}^N$, Equ. \eqref{start} is not invariant under translations.
 Furthermore,
owing to the mass constraint,  the reflection technique, which is typically used in half-space problems (see \cite{Cerami-half}) to compare the values of the energy functional with those of its limiting counterpart, proves ineffective in our scenario. Additionally, the presence of the Hardy potential introduces a singularity at zero, further complicating the situation.

In this paper, we establish a suitable framework to explore the existence of normalized solutions for \eqref{start}. We begin by defining the Hilbert space
 $H_{+,\mu}:= H^1(\mathbb{R}_+^N)$ for $\mu\in [0, \bar{\mu})$
with the scalar product
$$\langle u,v \rangle_{H_{+,\mu}}= \int_{\mathbb{R}_+^N}\left[\nabla u \nabla v -\frac{\mu}{|x|^2}uv + uv \right]dx
$$ and the norm
$\|u\|_{H_{+,\mu}}=\sqrt{\langle u,u \rangle_{H_{+,\mu}}}$. In \cite{LZ2023}, it was shown that, by the classical Hardy inequality, the norm $\|\cdot\|_{H_\mu}$ in $\mathbb{R}^N$ involving the Hardy potential is equivalent to $\|\cdot\|_{H^1(\mathbb{R}^N)}$. However, for the half-space problem with Neumann boundary conditions, such a Hardy-type inequality is unavailable.  To overcome this obstacle,  we first establish an appropriate Hardy inequality in $H^1(\mathbb{R}^N_+)$.
Specifically, we prove that there exists a constant $\delta>0$ (independent of $u$) such that for $\mu>0$ small enough, the following inequality holds for all $u\in H^1(\mathbb{R}_+^N)\setminus\{0\}$:
\[0<\min\{1-\frac{\mu}{\bar{\mu}}-\delta \mu, 1-\mu C(\delta) \}\int_{\mathbb{R}_+^N} \left(|\nabla u|^2 + u^2 \right)dx\leq \int_{\mathbb{R}_+^N}\left( |\nabla u|^2- \frac{\mu}{|x|^2}u^2 + u^2 \right)dx,\]
where $C(\delta)$ is a constant defined in Lemma \ref{hardy}. This result implies that $\|\cdot\|_{H_{+,\mu}}$ is equivalent to the standard norm $\|\cdot\|_{H^1(\mathbb{R}^N_+)}$, allowing us to proceed with our analysis. We denote the space of radial functions by $H_{r, \mu}:=H_{rad}^1(\mathbb{R}_+^N)$, which is equipped with the same scalar product and norm as those in $H_{\mu}$.

Next, we introduce the energy functional $J: H_{+,\mu}\to \mathbb{R}$ defined by
 $$
J(u)=\frac{1}{2}\int_{\mathbb{R}_+^N}|\nabla u|^2 dx-
\frac{1}{2}\int_{\mathbb{R}_+^N}\frac{\mu}{|x|^2}u^2 dx- \frac{1}{q}\int_{\mathbb{R}_+^N}|u|^q dx-\frac{1}{p}\int_{\mathbb{R}_+^N}|u|^p dx
$$
and the constraint $\mathcal{S}_{+,a}$ by
$$
\mathcal{S}_{+,a}:=\left \{u\in H_{+,\mu}:~\int_{\mathbb{R}_+^N}|u|^2 dx=a \right \}.
$$
We also define the subset of radial functions in $\mathcal{S}_{+,a}$ by $\mathcal{S}_{r,a}=\mathcal{S}_{+,a} \cap H_{r,\mu}$.

Our main theorem is as follows.
  \begin{theorem}\label{mainhalf}
  	Let $N\geq 3$, $q\in (2, 2_*)$, $p\in (2, 2^*)$.
  	There exists a positive constant $a_*>0$ such that for any $a\in (0, a_*)$, there exists  $\mu_a\in (0, \frac{\bar{\mu}}{4})$ 
  	such that when $\mu\in (0, \mu_a)$, problem \eqref{start} admits two normalized solutions,  one of which is a local minimizer of $J|_{\mathcal{S}_{+,a}}$, and another one is of mountain pass type.
    \end{theorem}

To find the first normalized solution, we focus on the minimization problem  
$$
m_a=\inf\limits_{u\in A_{k,a}}J(u),
$$ 
where $A_{k,a}=\{u\in \mathcal{S}_{+,a}: \|\nabla u\|_2^2
\leq k\}$ for some appropriately chosen $k>0$.  We consider a minimizing sequence $\{u_n\}\subset A_{k,a}$ of $J$. By utilizing the decreasing rearrangement and 
Ekeland's variational principle, we get a new minimizing sequence, still denoted by $\{u_n\}$, within $A_{k,a} \cap H_{r,\mu}$, which also constitutes a Palais-Smale (PS) sequence for $J$. By analyzing the relationship between $\{\lambda_n\}$ and $m_a$, we can establish the strong convergence of $\{u_n\}$ in $H_{r,\mu}$.

For the second normalized solution, which features mountain pass characteristics, the main challenge lies in establishing a bounded Palais-Smale (PS) sequence. Traditionally, constructing a (PS) sequence for combined nonlinearities diverges from the arguments used for purely $L^2$-supercritical problems, as seen in \cite{Jeanjean1997,Jeanjean-Lu}. Instead, one typically adopts the approach outlined in \cite{Soave-2020JDE,Soave-2020JFA}, which employs Ghoussoub's minimax method in conjunction with the decomposition of Pohozaev manifolds, as exemplified in \cite{JL2022,LZ2023}. Moreover, the sign of the Lagrangian multiplier is crucial in compactness arguments, and it has been commonly established as positive through the use of the Pohozaev identity in prior studies.

In our work, drawing inspiration from \cite{BCJS2023Nolinearity,CJS-2024}, we introduce a novel approach that combines a parameterized minimax principle with Morse index information to secure a bounded (PS) sequence. Notably, our analysis of the Morse index of the (PS) sequence demonstrates that the Lagrangian multiplier is nonnegative before proving it to be positive, a step that was challenging with previous methodologies. It's important to note that our method necessitates only the presence of mountain pass geometry on the $L^2$-sphere, rendering our approach more straightforward and universally applicable to the study of $L^2$-supercritical problems with a variety of nonlinearities.

 Furthermore, the compactness of the sequence $\{u_n\}$ is more involved than in previous arguments. More precisely,
using the above Hardy inequality in $H^1(\mathbb{R}^N_+)$, we can only ensures that 
\begin{eqnarray*}
\int_{\mathbb{R}_+^N}\frac{1}{|x|^2}u^2 dx~~\mbox{is bounded by}~~~ \int_{\mathbb{R}_+^N}\left[\left(\frac{1}{\bar{\mu}}+\delta\right) |\nabla u|^2 + C(\delta) u^2\right] dx.
\end{eqnarray*}
Given the presence of a sequence of bounded Lagrange multipliers $\{\lambda_n\}\subset \mathbb{R}$, denoting $\lambda_{0}=\lim\limits_{n\to+\infty}\lambda_n$, it is crucial to rule out the possibility that $\lambda_0>0$ is small as $\mu>0$ is sufficiently small. To address this issue, we utilize an auxiliary function $h(t)$ to demonstrate that the local minimum geometry of $J$ holds for any $a\in (0, a_*)$, where $a_*$ is small and independent of $\mu$. Consequently, together with the nonnegative of $\lambda_n$ and a Pohozaev identity involving the Hardy potential on $\mathbb{R}_+^N$ established in Section 2, we establish a uniform lower bound away from zero for $\{\lambda_n\}$, which in turn enables us to recover the compactness of $\{u_n\}$ for small $\mu$.

The paper is organized as follows. In Section 2, we establish both a Hardy inequality in $H_{+,\mu}$ as well as a Pohozaev identity in $\mathbb{R}_+^N$, and discuss the properties of the auxiliary function $h(t)$.
In Section 3, we demonstrate the existence of a local minimizer of $J_\rho$ on $\mathcal{S}_{+,a}$.  
Section 4 begins by proving the existence of a mountain pass critical point $u_\rho$ with Morse index information for the parameterized functional $J_\rho$ on the set $S_{r,a}$ for almost every $\rho\in[\frac{1}{2},1]$. Subsequently, we analyze the original problem and complete the proof of Theorem \ref{mainhalf}. For notational convenience, we denote the norm of the space $L^r(\mathbb{R}_+^N)(1\le r\le +\infty)$ by $\|\cdot\|_r$, and the norm in $H_{+,\mu}$, as well as $H_{r,\mu}$, by $\|\cdot\|$. Additionally, $B_r(0)$ represents a ball in $\mathbb{R}^N$ centered at the origin with radius $r>0$.


\section{Preliminaries}\label{constant}
\noindent

We recall the following Gagliardo-Nirenberg inequality \cite[Theorem 5.8]{Adams}:
for every $N\geq 1$ and $p\in(2,2^*)$, there exists a constant $C_{N,p}$ depending on $N$
and $p$ such that 
$$\|u\|_p\leq C_{N,p}\|u\|_2^{1-\gamma_p}\|u\|^{\gamma_p},$$
where 
$\gamma_p:=\frac{N(p-2)}{2p}$.

The classical Hardy inequality (refer to \cite[Corollary 2.2]{AA-1998}) states that, for a unit ball $B_1(0)\subset \mathbb{R}^N$, where $N\geq 3$, if $u\in H^1_0(B_1(0))$, then $\frac{u}{|x|}\in L^2(B_1(0))$ and 
\begin{equation}\label{eqqqq}
	\int_{B_1(0)} \frac{u^2}{|x|^2} dx\leq \frac{1}{\bar{\mu}}\int_{B_1(0)} |\nabla u|^2 dx.
\end{equation}
However, this inequality is only valid for the Dirichlet boundary case.
Motivated by \cite{Chabrowski,Han}, we build a Hardy inequality on $H_{+,\mu}$ as follows.

\begin{lemma}\label{hardy}
	For every $\delta>0$ there exists a constant $C=C(\delta)>0$ such that 
	\[\int_{\mathbb{R}_+^N}\frac{1}{|x|^2}u^2 \leq (\frac{1}{\bar{\mu}}+\delta) \int_{\mathbb{R}_+^N}|\nabla u|^2 + C(\delta) \int_{\mathbb{R}_+^N}u^2 \]
	for every $u\in H_{+,\mu}$.
\end{lemma}
\begin{proof}
Set $B_{r}^+(0)= B_r(0)\cap \mathbb{R}_+^N$.
We define a $C^1$ cut-off function $\phi(x)$ such that $|\phi(x)|\leq 1$ for all $ x\in B_1(0)$, and
\begin{equation*}
	\phi(x)=\begin{cases}
1 ~~\mbox{for}~~x\in B_\frac{1}{2}(0),\\
0 ~~\mbox{for}~~x\in \mathbb{R}^N\backslash B_1(0).
	\end{cases}
\end{equation*}
Moreover, we can assume that $\phi(x)$ is even with respect to $x_ N$ and that $\phi$ satisfies 
$|\nabla \phi(x)|\leq 4$ for $x\in B_1(0)\backslash B_\frac{1}{2}(0)$. Then 
\begin{equation*}\label{9-30-1}
	\int_{\mathbb{R}^N_+} \frac{u^2}{|x|^2} dx = \int_{B_1^+(0)} \frac{\phi^2 u^2}{|x|^2} dx + \int_{\mathbb{R}^N_+} \frac{u^2(1-\phi^2)}{|x|^2} dx.
\end{equation*}

Next, we extend $u$ by reflection across the plane $\{x_N=0\}$. 
      Specifically, for $(x',x_N)\in \mathbb{R}^N$, we define
        \begin{equation*}
        	\tilde{u}(x',x_N)=
        	\begin{cases}
        		&u(x',x_N) ~{\rm if}~x_N\geq 0,\\
        		&u(x',-x_N)~{\rm if}~x_N<0.
        	\end{cases}
        \end{equation*}
We can verify that $\phi \tilde{u}\in H^1_0(B_1(0))$  and that $\phi \tilde{u}$ even in $x_N$.
By \eqref{eqqqq}, we have
\begin{align*}
	\int_{B_1^+(0)} \frac{\phi^2u^2}{|x|^2} dx= \frac{1}{2}\int_{B_1(0)} \frac{\phi^2 \tilde{u}^2}{|x|^2} dx \leq \frac{1}{2\bar{\mu}}\int_{B_1(0)} |\nabla (\phi \tilde{u})|^2 dx=\frac{1}{\bar{\mu}}\int_{\mathbb{R}_+^N}|\nabla (\phi u)|^2 dx.
\end{align*}
According to the Young's inequality, for any $\delta>0$, we get
\begin{align*}
	\frac{1}{\bar{\mu}}\int_{\mathbb{R}_+^N}|\nabla (\phi u)|^2 dx &=
	\frac{1}{\bar{\mu}}\int_{\mathbb{R}_+^N}\left[  u^2|\nabla \phi|^2 +2u\phi \nabla u\nabla\phi +\phi^2|\nabla u|^2\right] dx\\
	&\leq (\frac{1}{\bar{\mu}}+\delta)\int_{\mathbb{R}_+^N}\phi^2|\nabla u|^2 dx+ (\frac{1}{\bar{\mu}}+ \frac{4}{\bar{\mu}^2\delta})\int_{\mathbb{R}_+^N} u^2|\nabla \phi|^2 dx.
\end{align*}
Thus, we deduce that
\begin{align*}
\int_{\mathbb{R}^N_+} \frac{|u|^2}{|x|^2} dx &\leq (\frac{1}{\bar{\mu}}+\delta)\int_{\mathbb{R}_+^N}\phi^2|\nabla u|^2 dx+ (\frac{1}{\bar{\mu}}+ \frac{4}{\bar{\mu}^2\delta})\int_{\mathbb{R}_+^N} u^2|\nabla \phi|^2 dx +\int_{\mathbb{R}_+^N} \frac{u^2(1-\phi^2)}{|x|^2} dx\\
    &\leq (\frac{1}{\bar{\mu}}+\delta)\int_{\mathbb{R}_+^N}|\nabla u|^2 dx+ (\frac{16}{\bar{\mu}}+\frac{64}{\bar{\mu}^2\delta}+4)\int_{\mathbb{R}_+^N}|u|^2 dx.
\end{align*}
This completes the proof.
\end{proof}

In the following, we shall prove a Pohozaev identity in the half space $\mathbb{R}^N_+$.
\begin{lemma}\label{poho}
	Assume that $\mu\in (0,\bar{\mu})$ and $g\in C^1(\mathbb{R}, \mathbb{R})$ satisfies $g(0)=0$ and $\lim\limits_{|s|\to \infty}\frac{g(s)}{|s|^{2^*-2}s}=0$.
	Let $u\in H_{+,\mu}$ be a solution to the following boundary value problem:
\begin{eqnarray}\label{madaua}
\begin{cases}
	&-\Delta u -\frac{\mu}{|x|^2}u =g(u)~~\mbox{in} ~~\mathbb{R}_+^N,\\
    &\frac{\partial u}{\partial \nu}=0 ~~ \mbox{on}~~ \partial\mathbb{R}_+^N.
\end{cases}
\end{eqnarray}
Then, we have
	\[\frac{N-2}{2}\int_{\mathbb{R}_+^N}\left( |\nabla u|^2-\frac{\mu}{|x|^2}u^2\right) dx - N\int_{\mathbb{R}_+^N}G(u) dx=0,\]
where $G(s)=\int_{0}^s g(\tau) d\tau$.
\end{lemma}

\begin{proof}
Firstly, we can verify that, for any $R>0$, there exists $C>0$ such that
	\[\left| \frac{g(u)}{u}+\frac{\mu}{|x|^2} \right|\leq C+C|u|^{\frac{4}{N-2}}, ~~\forall x\in \mathbb{R}_+^N\backslash B_R^+(0), \]
where $B_R^+(0):=\{x\in \mathbb{R}_+^N : |x|<R \}$.
By similar arguments as in \cite{Brezis-Kato}, we deduce that
$u\in L^p_{loc}(\mathbb{R}_+^N\backslash B_R^+(0))$, for any $p\in [1, \infty)$.
Using standard $L^p$ estimate, it follows that $u\in W^{2,p}_{loc}(\mathbb{R}_+^N\backslash B_R^+(0))$ for any $1\le p< \infty$. Then, by letting $R\to 0$ we get $u\in W^{2,p}_{loc}(\mathbb{R}_+^N\backslash \{0\}$ for any $1\le p<\infty$, by which we get $u\in C^{1,\alpha}(\mathbb{R}_+^N\backslash \{0\})$, where $\alpha\in (0,1)$.

For any $R>0$ and $\sigma\in (0, R)$, we define 
\begin{align*}
	&D_{\sigma, R}:=\{x=(x', x_N)\in \mathbb{R}^{N-1}\times [\sigma, +\infty): |x|\leq R^2 \},\\
	&\partial D_{\sigma, R}^1:= \{x=(x', x_N)\in \mathbb{R}^{N-1}\times \{x_N= \sigma\}: |x'|^2\leq R^2-\sigma^2 \},\\
	&\partial D_{\sigma, R}^2:= \{x=(x', x_N)\in \mathbb{R}^{N-1}\times [\sigma, +\infty): |x|^2= R^2\}.
\end{align*}
Clearly, $\partial D_{\sigma, R}=\partial D_{\sigma, R}^1\cup \partial D_{\sigma, R}^2$.

In what follows, we derive the Pohozaev identity on $D_{\sigma, R}$. To do this, we multiply
 \eqref{madaua} 
 by $x\cdot \nabla u$ and apply the
integration by parts formula (refer to \cite[Theorem 5.2]{LLT2021}):
\begin{equation*}
	\int_{D_{\sigma, R}}g(u)(x\cdot \nabla u) dx=\int_{D_{\sigma, R}}x\cdot \nabla G(u) dx= -N \int_{D_{\sigma, R}} G(u)dx+ \int_{\partial D_{\sigma, R} }G(u) x\cdot \nu dS
\end{equation*}
and
\begin{equation*}
	2\int_{D_{\sigma, R}}\frac{u}{|x|^2} (x\cdot \nabla u)dx= \int_{D_{\sigma, R}}\frac{x\cdot \nabla u^2}{|x|^2}dx =-(N-2)\int_{D_{\sigma, R}} \frac{u^2}{|x|^2}dx+ \int_{\partial D_{\sigma, R}}\frac{u^2}{|x|^2} x\cdot \nu dS.
\end{equation*}
Similarly, we obtain
\begin{equation*}
	\int_{D_{\sigma, R}}(x\cdot \nabla u)\Delta u dx=
	\frac{N-2}{2}\int_{D_{\sigma, R}} |\nabla u|^2dx+\int_{\partial D_{\sigma, R}}\left((x\cdot \nabla u)\nabla u-\frac{|\nabla u|^2}{2}x\right)\cdot \nu dS.
\end{equation*}
Thus, we arrive at
\begin{equation}\label{feifei}
	\begin{split}
		&\frac{2-N}{2}\int_{D_{\sigma, R}} \left(|\nabla u|^2-\frac{\mu }{|x|^2}u^2\right) dx +N\int_{D_{\sigma, R}} G(u)dx\\
		&=\int_{\partial D_{\sigma, R}} \left[ \left((x\cdot \nabla u)\nabla u-\frac{|\nabla u|^2}{2}x\right)\cdot \nu + \frac{\mu}{2}\frac{u^2}{|x|^2} x\cdot \nu+ G(u) x\cdot \nu \right] dS. 
	\end{split}
\end{equation}

Next, we demonstrate that the boundary term converges to $0$ as $\sigma\to 0$ and $R\to \infty$. Since $u$ satisfies the Neumann boundary condition, we deduce, for any $R>0$, that
\begin{align*}
	&\int_{\partial D_{\sigma, R}}\left[ \left((x\cdot \nabla u)\nabla u-\frac{|\nabla u|^2}{2}x\right)\cdot \nu +\frac{\mu}{2}\frac{u^2}{|x|^2} x\cdot \nu+ G(u) x\cdot \nu  \right]dS\\
	&= \int_{\partial D_{\sigma, R}^1}\left[(x\cdot \nabla u)(-\partial_{x_N}u)+\frac{\sigma|\nabla u|^2}{2}-\frac{\mu\sigma}{2}\frac{ u^2}{|x|^2}- \sigma G(u)\right]dS\\
	&+\int_{\partial D_{\sigma, R}^2}\frac{R}{2}\left(|\nabla u|^2+ \frac{\mu}{|x|^2}u^2+ 2G(u)\right) dS\\
&\to \int_{\partial B_R^{2,+}}\frac{R}{2}\left(|\nabla u|^2+ \frac{\mu}{|x|^2}u^2+ 2G(u)\right) dS~~\mbox{as}~~\sigma\to0^+,
\end{align*}
where $\partial B_R^{2,+}:=\mathbb{R}^N_+\cap \partial B_R(0)$.

Denote $A(x, u)=|\nabla u|^2+\frac{\mu}{|x|^2}u^2+2|G(u)|$. We claim that there exists a sequence $\{R_n\}\subset \mathbb{R}^+$ such that 
\[\lim_{n\to \infty} R_n\int_{\partial B_R^{2,+}} A(x,u) dS=0~~\mbox{as}~~R_n\to \infty.\]
If not, we may assume that 
\[\liminf_{R\to \infty}R\int_{\partial B_R^{2,+}} A(x,u) dS\geq c\]
for some $c>0$. Then, there exists a constant $R^*>0$ such that
\[\int_{\partial B_R^{2,+}} A(x,u) dS\geq \frac{c}{R}, ~~\forall R\geq R^*,\]
which implies that 
\begin{align*}
\int_{\mathbb{R}_+^N}A(x,u)dx
	\geq \int_{R^*}^{\infty}\int_{\partial B_R^{2,+}}A(x,u)dS dR
	\geq c \ln{\frac{R}{R^*}}\to \infty~~\mbox{as}~~R\to \infty,
\end{align*}	
which is a contradiction to Lemma \ref{hardy} and the fact that $u\in H_{+,\mu}$. Thus we can derive the Pohozaev identity from \eqref{feifei}.
\end{proof}

Let $\delta=\frac{1}{\bar{\mu}}$ and $C_{\bar{\mu}}=\frac{80}{\bar\mu}+4$.
By the Gagliardo-Nirenberg inequality and Lemma \ref{hardy}, for every $u\in \mathcal{S}_{+,a}$, we have
\begin{equation*}
	\begin{split}
	  J(u)
	  &\geq \frac{1}{2}(1-\frac{2\mu}{\bar{\mu}})\int_{\mathbb{R}_+^N}|\nabla u|^2 dx-\frac{\mu C_{\bar{\mu}}}{2}\int_{\mathbb{R}_+^N}u^2 dx- \frac{1}{q}\int_{\mathbb{R}_+^N}|u|^q dx-\frac{1}{p}\int_{\mathbb{R}_+^N}|u|^p dx\\
	  &\geq \frac{1}{2}(1-\frac{2\mu}{\bar{\mu}})\int_{\mathbb{R}_+^N}|\nabla u|^2 dx-\frac{\mu C_{\bar{\mu}}}{2}\int_{\mathbb{R}_+^N}u^2 dx-\frac{C_{N,q}^q}{q}a^{\frac{q(1-\gamma_q)}{2}}\|u\|^{q\gamma_q}-\frac{C_{N,p}^p}{p}a^{\frac{p(1-\gamma_p)}{2}}\|u\|^{p\gamma_p}\\
	  &=\frac{1}{2}(1-\frac{2\mu}{\bar{\mu}})\|u\|^2- \frac{C_{N,q}^q}{q}a^{\frac{q(1-\gamma_q)}{2}}\|u\|^{q\gamma_q}-\frac{C_{N,p}^p}{p}a^{\frac{p(1-\gamma_p)}{2}}\|u\|^{p\gamma_p}-(\frac{1}{2}-\frac{\mu}{\bar{\mu}}+\frac{\mu C_{\bar{\mu}}}{2})a.
	\end{split}
\end{equation*}

We define a function $h:[0, +\infty) \to \mathbb{R}$ by
$$
h(t):=\frac{1}{2}(1-\frac{2\mu}{\bar{\mu}})t^2-\frac{C_{N,p}^p}{p}a^{\frac{p(1-\gamma_p)}{2}}t^{p\gamma_p}-\frac{C_{N,q}^q}{q}a^{\frac{q(1-\gamma_q)}{2}}t^{q\gamma_q}
-(1-\frac{\mu}{\bar{\mu}}+\frac{\mu C_{\bar{\mu}}}{2})a.
$$

\begin{lemma}\label{aaa}
There exists a constant $a_*>0$ such that for any $a\in (0, a_*)$ and any $\mu\in (0, \frac{\bar{\mu}}{4})$, the function $h(t)$ attains a unique positive global maximum in the integral $(a^{\frac{1}{2}}, +\infty)$.
\end{lemma}
\begin{proof}
For $t>0$, we have $h(t)>0$ if and only if 
\begin{eqnarray}\label{10-12-1}
R(t)t^{q\gamma_q}>\frac{C_{N,q}^q}{q}a^{\frac{q(1-\gamma_q)}{2}}t^{q\gamma_q}+(\frac{1}{2}-\frac{\mu}{\bar{\mu}}+\frac{\mu C_{\bar{\mu}}}{2})a
\end{eqnarray}
where 
\[R(t)=\frac{1}{2}(1-\frac{2\mu}{\bar{\mu}})t^{2-q\gamma_q}-\frac{C_{N,p}^p}{p}a^{\frac{p(1-\gamma_p)}{2}}t^{p\gamma_p-q\gamma_q}.\]
We now consider $R(t)$.
 Direct calculation yields 
\begin{equation*}
	R'(t)=\frac{1}{2}(1-\frac{2\mu}{\bar{\mu}})(2-q\gamma_q)t^{1-q\gamma_q}-\frac{C_{N,p}^p}{p}a^{\frac{p(1-\gamma_p)}{2}}(p\gamma_p-q\gamma_q)t^{p\gamma_p-q\gamma_q-1}.
\end{equation*}
Set $$
\eta_1=\frac{p(1-\frac{2\mu}{\bar{\mu}}) (2-q\gamma_q)}{2C_{N,p}^p(p\gamma_p-q\gamma_q)},~~~~
\eta_2=\frac{(1-\frac{2\mu}{\bar{\mu}})(p\gamma_p-2)}{2(p\gamma_p-q\gamma_q)}.
$$
We find that $R(t)$ has a global maximum at the point 
\begin{equation*}\label{uijin}
	\bar{t}=
\eta_1^{\frac{1}{p\gamma_p-2}}
a^{-\frac{p(1-\gamma_p)}{2(p\gamma_p-2)}},
\end{equation*}
and at this point, we have
\[
R(\bar{t})=
\eta_1^{\frac{2-q\gamma_q}{p\gamma_p-2}}\eta_2 a^{-\frac{p(1-\gamma_p)(2-q\gamma_q)}{2(p\gamma_p-2)}}.
\]
Define 
$$
\theta_1=1+ \frac{p(1-\gamma_p)}{p\gamma_p-2},~~~~a_1=\eta_1^{\frac{2}{\theta_1 (p\gamma_p-2)}}.
$$
It can be readily observed that $\bar{t}>a^{\frac{1}{2}}$ for all $a\in (0, a_1)$.

Next, denote \[\theta_2= \frac{q(1-\gamma_q)}{2}+ \frac{p(1-\gamma_p)(2-q\gamma_q)}{2(p\gamma_p-2)},~ ~ \theta_3=\frac{p-2}{p\gamma_p-2}.\]
Assuming $h(\bar{t})>0$, by (\ref{10-12-1}) and the value of $C_{\bar{\mu}}$ from Lemma \ref{hardy} we obtain 
\begin{equation}\label{bao}
\frac{C_{N,q}^q}{q}\eta_1^{\frac{q\gamma_q}{p\gamma_p-2}}a^{\theta_2}+(1+\frac{39\mu}{\bar{\mu}}+2\mu)a^{\theta_3}<\eta_1^{\frac{2}{p\gamma_p-2}}\eta_2.
\end{equation}
In view of $\mu\in (0, \frac{\bar{\mu}}{4})$, we find  
$$
0<\frac{p (2-q\gamma_q)}{4C_{N,p}^p(p\gamma_p-q\gamma_q)}\le \eta_1\le \frac{p (2-q\gamma_q)}{2C_{N,p}^p(p\gamma_p-q\gamma_q)}
$$
and
$$
0<\frac{p\gamma_p-2}{4(p\gamma_p-q\gamma_q)}\le \eta_2\le \frac{p\gamma_p-2}{2(p\gamma_p-q\gamma_q)}.
$$
Since all the constants $\theta_1, \theta_2, \theta_3$ are positive and independent of $\mu$, there exists $a_2>0$, also independent of $\mu$, such that \eqref{bao} holds for $a\in (0,a_2)$. 
Let 
$$
\tilde{a}_1=\left(\frac{p (2-q\gamma_q)}{4C_{N,p}^p(p\gamma_p-q\gamma_q)}\right)^{\frac{2}{\theta_1 (p\gamma_p-2)}}.
$$
Then, $\tilde{a}_1\in (0, a_1)$ and is dependent of $\mu$ such that $\bar{t}>a^{\frac{1}{2}}$ for all $a\in (0, \tilde{a}_1)$.

Take $a_*=\min\{\tilde{a}_1, a_2\}$. By fundamental analysis we know that the function $h'(t)$ has exactly one zero for $t>0$. Therefore, for $a\in (0,a_*)$, the function $h(t)$ has a unique positive global maximum in the integral $(a^{\frac{1}{2}}, +\infty)$.
\end{proof}

\section{Existence of local minimizers}
In this section, we apply Ekeland variational principle to establish the existence of local minimum normalized solutions.

For $a\in (0, a_*)$, take $k=\bar{t}^2-a>0$ and define the set $A_{k,a}=\{u\in \mathcal{S}_{+,a}: \|\nabla u\|_2^2
\leq k\}$. We consider the following local minimization problem: for any $a\in (0, a_*)$,
\[m_a:= \inf_{u\in A_{k,a}}J(u).\]
Our goal is to find a local minimizer of $J$ in the set $A_{k,a}$.
\begin{lemma}\label{exist}
Assume that $a\in (0, a_*)$. There exists $\mu_a\in (0, \frac{\bar{\mu}}{4})$ such that for any $\mu\in (0,\mu_a)$, $m_a$ is achieved by some $u_0\in A_{k,a}$ which satisfies
\begin{eqnarray*}
m_a=J (u_0)<0<\frac{a}{2}\le \inf_{u\in \partial A_{k,a}}J (u).
\end{eqnarray*}
\end{lemma}
\begin{proof}
Fix  $a\in (0, a_*)$.
By Lemma \ref{aaa}, it follows that $J(u)\ge\frac{a}{2}$ for all $u\in \partial A_{k,a}$.
Now, for any $u\in \mathcal{S}_{+,a}$, define
\[s\star u(x):= e^{\frac{Ns}{2}}u(e^s x),~~s\in \mathbb{R}.\]
We can then get
\begin{equation}\label{zhuguang}
	J(s\star u)\leq \frac{e^{2s}}{2}\int_{\mathbb{R}_+^N}|\nabla u|^2 dx-\frac{e^{q\gamma_q s}}{2q}\int_{\mathbb{R}_+^N}|u|^q dx-\frac{e^{p\gamma_p s}}{2p}\int_{\mathbb{R}_+^N}|u|^p dx.
\end{equation}
Taking into account the fact that $2<q<2+\frac{4}{N}<p<2^*$, we see that 
$J(s\star u)\to 0^-$ as $s\to -\infty$.
Therefore, there exists sufficiently negative $s_0<0$ such that $\|\nabla (s_0\star u)\|_2^2= e^{2s_0}\|\nabla u\|_2^2<\bar{t}$ and $J(s_0\star u)<0$. This implies that $m_a<0$.
 
Considering a minimizing sequence $\{v_n\}$ of $J$ in $A_{k,a}$ at the level $m_a$. 
Without loss of generality, we can replace $v_n$ by $|v_n|$ to ensure that $v_n\geq 0$.  
 By applying the decreasing rearrangement and 
 Ekeland’s variational principle, we can obtain a new minimizing sequence $\{u_n\}\subset A_{k,a}$ for $m_a$. This new sequence
$\{u_n\}$ satisfies the property that $\|u_n- v_n\|\to 0$ as $n\to \infty$,  and it forms a Palais-Smale sequence for $J$ on $\mathcal{S}_{+,a}$. It is easily seen that $||u_n||^2\leq a+k$,  which implies that $\{u_n\}$ is bounded. By standard arguments, there exists a sequence $\{\lambda_n\}\subset \mathbb{R}$
such that
\begin{equation}\label{weakcon}
	J'(u_n)\varphi+ \lambda_n \int_{\mathbb{R}_+^N}u_n\varphi dx\to 0
~~ \mbox{as}~~n\to \infty,~\forall \varphi\in H_{+,\mu}.
\end{equation}
We also deduce that there exists $u_0\in H_{r,\mu}$ such that, up to a subsequence,
$u_n\rightharpoonup u_0$ weakly in $H_{+,\mu}$ and $u_n \rightarrow u_0$ strongly in
$L^r(\mathbb{R}_+^N)$ for $r\in(2, 2^*)$.

Taking $u_n$ as the test function in (\ref{weakcon}), we obtain
\begin{equation}\label{boundness}
	\lambda_n a=-\int_{\mathbb{R}_+^N}\left(|\nabla u_n|^2-\frac{\mu}{|x|^2}u_n^2 \right)dx +\int_{\mathbb{R}_+^N}\left(|u_n|^q+|u_n|^p\right) dx.
\end{equation}
The boundness of $\{ u_n \}$ and \eqref{boundness} imply that
$\{\lambda_n\}$ is bounded as well.
Thus, up to a subsequence, we have $\lambda_n\to \lambda_*\in \mathbb{R}$.

We claim that $\lambda_*>0$. Indeed, in view of $J(u_n)=m_a+o_n(1)$, together with \eqref{boundness} we deduce that
\begin{align*}
 2m_a+o(1)&=2J(u_n)-\big[J'(u_n)u_n+\lambda_n a\big]\\
	&= -\lambda_n a+ \frac{ (q-2)}{q}\int_{\mathbb{R}_+^N}|u_n|^q dx+ \frac{ (p-2)}{p}\int_{\mathbb{R}_+^N}|u_n|^p dx\geq -\lambda_n a.
\end{align*}
Thus, $\lambda_n \geq -\frac{2m_a}{a}>0$.
Up to a subsequence,
\begin{eqnarray*}\label{lambdabdd}
	\lambda_n\to \lambda_*\geq -\frac{2m_a}{a} >0 ~~\mbox{as}~~n\to \infty.
\end{eqnarray*}
By weak convergence and \eqref{weakcon} it follows that
\begin{equation}\label{weaklim}
	J'(u_0)\varphi-\lambda_* \int_{\mathbb{R}_+^N}u_0\varphi dx = 0
\end{equation}
for every $\varphi \in H_{+,\mu}$.
Choosing $\varphi= u_n-u_0$ in \eqref{weakcon} and \eqref{weaklim} and subtracting, we can prove
\begin{equation*}
	\int_{\mathbb{R}_+^N}\left[|\nabla (u_n-u_0)|^2-\frac{\mu}{|x|^2}|u_n-u_0|^2 \right]dx +\lambda_*\int_{\mathbb{R}_+^N}|u_n-u_0|^2 dx=o_n(1).
\end{equation*}
By Lemma \ref{hardy}, taking $\mu_a\in (0, \frac{\bar{\mu}}{4})$ such that $\mu_a C_{\bar{\mu}}\leq -m_a$. Then, for $\mu \in (0, \mu_a)$, we have
\[
\frac{1}{2}\int_{\mathbb{R}_+^N}|\nabla (u_n-u_0)|^2 dx+ \frac{\lambda_*}{2}\int_{\mathbb{R}_+^N}(u_n-u_0)^2 dx\leq o_n(1).\]
This establishes the strong convergence of $\{u_n\}$ in $H_{+,\mu}$. 
\end{proof}

\section{Existence of mountain pass type solutions}

This section is devoted to establishing the existence of mountain pass type normalized solutions. 

\subsection{Approximating problems}

We introduce a family of functionals $J_\rho: H_{r,\mu}\to \mathbb{R}$ defined by
\begin{equation*}\label{para}
	J_{\rho}(u)=\frac{1}{2}\int_{\mathbb{R}_+^N}|\nabla u|^2 dx- \frac{1}{2}\int_{\mathbb{R}_+^N}\frac{\mu}{|x|^2}u^2 dx- \frac{\rho}{q}\int_{\mathbb{R}_+^N}|u|^q dx-\frac{\rho}{p}\int_{\mathbb{R}_+^N}|u|^p dx,~~\forall \rho\in \left[\frac{1}{2}, 1\right].
\end{equation*}
By the arguments in the previous section, it follows that for any $a\in (0, a_*)$ and any $u\in \partial A_{k,a}$, we have $J_\rho (u)\geq J(u) \geq h(\|u\|)\ge \frac{a}{2}>0$, for any $\rho \in [\frac{1}{2}, 1]$.

For any $a\in (0, a_*)$, we consider 
\[m_\rho(a)= \inf_{u\in A_{k,a}}J_\rho(u).\]
Similar to Lemma \ref{exist}, for any $\mu\in (0, \mu_a)$, $m_\rho(a)$ is achieved by some $u_{0,\rho}\in A_{k,a}$ which satisfies
\begin{eqnarray*}\label{Estar}
m_\rho(a)=J_\rho (u_{0,\rho}) =\inf_{u\in A_{k,a}} J_\rho(u)<0<\frac{a}{2}\le \inf_{u\in \partial A_{k,a}}J_\rho (u),~~\forall \rho\in \left[\frac{1}{2}, 1 \right].
\end{eqnarray*}
At this point, we are ready to show that the functional $J_\rho$ admits a mountain pass geometry uniformly for $\rho \in [\frac{1}{2}, 1]$ on the constraint $\mathcal{S}_{r,a}$.
\begin{lemma}\label{MP}
Assume that $a\in (0, a_*)$.
There exist $w_1, w_2\in \mathcal{S}_{r,a}$ independent of $\mu$ such that
\[c_\rho:= \inf_{\gamma\in \Gamma}\max_{t\in[0,1]}J_\rho(\gamma(t))>\max\{J_\rho(w_1), J_\rho(w_2) \} ,~~ \forall \rho\in \left[\frac{1}{2},1 \right], \]
where 
	\begin{equation*}\label{Gamma}
		\Gamma:=\{\gamma\in C([0,1], \mathcal{S}_{r,a}) : \gamma(0)=w_1 , ~\gamma(1)=w_2 \}.
	\end{equation*}
\end{lemma}

\begin{proof}
For any $v\in \mathcal{S}_{r,a}$, we have 
\[J_\rho(s\star v)\leq \frac{e^{2s}}{2}\int_{\mathbb{R}_+^N}|\nabla v|^2 dx-\frac{e^{q\gamma_q s}}{q}\int_{\mathbb{R}_+^N}|v|^q dx -\frac{e^{p\gamma_p s}}{p}\int_{\mathbb{R}_+^N}|v|^p dx.\]
According to $2<q<2+\frac{4}{N}<p<2^*$, it follows that
$J_\rho(s\star v)\to -\infty$ as $s\to +\infty$.
Therefore, there exists a sufficiently large $s_1>0$ independent of $\mu$  such that $\|\nabla (s_1\star v)\|_2^2= e^{2s_1}\|\nabla v\|_2^2> k$ and $J_\rho(s_1\star v)< 0$.
In particular, we get
\begin{equation*}
	J_\rho(s_1\star v)<J_1(u_{0,\frac{1}{2}})\leq J_\rho(u_{0,\frac{1}{2}})\leq J_{\frac{1}{2}}(u_{0,\frac{1}{2}})<0,  ~\forall \rho\in \left[\frac{1}{2},1 \right].
\end{equation*}
We choose $w_1=u_{0,\frac{1}{2}}$ and $w_2=s_1\star v$. Clearly, $u_{0,\frac{1}{2}}\in A_{k,a}, s_1\star v\not\in A_{k,a}$.
By continuity,
for any $\gamma \in \Gamma$, there exists $t_\gamma \in [0,1]$ such that $ \gamma(t_\gamma)\in \partial A_{k,a}$. Hence, by Lemma \ref{exist}, we deduce
\begin{align*}
c_{\rho}\ge J_\rho(\gamma(t_\gamma)) \geq  \inf_{u\in \partial A_{k,a}}J_\rho(u)\ge \frac{a}{2}>0=\max\{J_\rho(w_1), J_\rho(w_2)\},  ~\forall \rho\in \left[\frac{1}{2},1 \right].
\end{align*}
This completes the proof.
\end{proof}

\begin{remark}
By the selection of $w_1$, $w_2$, and $\Gamma$, these functions remain independent of the parameter $\rho$.
\end{remark}

\begin{lemma}\label{thank}
	For any $\lambda<0$, there exists a subspace $Y$ of $H_{r,\mu}$ with $dim Y=3$ such that 
	\begin{equation*}
		\int_{\mathbb{R}_+^N}|\nabla w|^2 dx +\lambda \int_{\mathbb{R}_+^N} |w|^2 dx\leq \frac{\lambda}{2}\int_{\mathbb{R}_+^N}\left(|\nabla w|^2 +w^2 \right)dx, ~~\forall w\in Y.
	\end{equation*}
\end{lemma}
\begin{proof}
	Take $\phi\in C_0^\infty(\mathbb{R}_+^N)$ with $supp \phi \subset B_1(0,\cdots,0,\frac{3}{2})$ such that $\int_{\mathbb{R}_+^N} |\phi|^2 dx=1$, where $B_1(0,\cdots,0,\frac{3}{2})$ is a unit ball centered at $(0,\cdots,0,\frac{3}{2})\in \mathbb{R}_+^N$.
	We define the following functions for $\eta>0$:
	\begin{align*}
		&\psi_{0,\eta}(x)=\eta^{\frac{N}{2}}\phi (\eta x', \eta x_N),\\
		&\psi_{1,\eta}(x)=\eta^{\frac{N}{2}}\phi (\eta x', \eta x_N- 4\eta ),\\
		&\psi_{2,\eta}(x)=\eta^{\frac{N}{2}}\phi (\eta x', \eta x_N- 8\eta ).
	\end{align*}
Clearly, for any $\eta>0$,
\begin{align*}
	\int_{\mathbb{R}_+^N} |\psi_{0, \eta}|^2 dx =
	\int_{\mathbb{R}_+^N} \eta^N \phi^2(\eta x', \eta x_N) dx
	=\int_{\mathbb{R}_+^N} \phi^2(y', y_N) dy=1.
\end{align*}
Similarly, we can check that $\int_{\mathbb{R}_+^N}|\psi_{1,\eta}|^2dx=\int_{\mathbb{R}_+^N}|\psi_{2,\eta}|^2dx=1$.
In particular, $\psi_{0,\eta}(x)$, $\psi_{1,\eta}(x)$ and $\psi_{2,\eta}(x)$ are orthogonal in $H_{r,\mu}$, as they have disjoint supports.
Set $w=\sum\limits_{j=0}^2 \tau_j \psi_{j,\eta}$. Then 
\begin{align*}
	\int_{\mathbb{R}_+^N}|\nabla w|^2 dx+ \lambda\int_{\mathbb{R}_+^N}|w|^2 dx
	&=\sum_{j=0}^2\int_{\mathbb{R}_+^N} \tau_j^2 |\nabla \psi_{j, \eta}|^2 dx+ \lambda \sum_{j=0}^2 \int_{\mathbb{R}_+^N} \tau_j^2 |\psi_{j, \eta}|^2 dx\\
	&=\eta^2\left(\sum_{j=0}^2\int_{\mathbb{R}_+^N} \tau_j^2 |\nabla \phi|^2 dx \right)+ \lambda \sum_{j=0}^2\tau_j^2  \int_{\mathbb{R}_+^N} |\phi|^2 dx\\
	&=(\eta^2 \xi +\lambda)\sum_{j=0}^2 \tau^2_j,
\end{align*}
where $\xi:= \int_{\mathbb{R}_+^N}|\nabla \phi|^2 dx >0$.
Similarly, $\|w\|_{H^1}^2=(\eta^2 \xi +1)\sum_{j=0}^2 \tau^2_j$.
Therefore,
\[\frac{\int_{\mathbb{R}_+^N}\left(|\nabla w|^2+\lambda|w|^2 \right)dx}{\|w\|_{H^1}^2}=\frac{\eta^2\xi+\lambda}{\eta^2\xi +1}\leq \frac{\lambda}{2} \]
for sufficiently small $\eta>0$, and for every $\tau_0, \tau_1, \tau_2\in \mathbb{R}$.
\end{proof}

For a domain $\mathbb{R}^N_+$ and $\phi, u\in H_{r,\mu}$, we consider
\begin{equation*}\label{morse}
	Q(\phi;u;\mathbb{R}^N_+):= \int_{\mathbb{R}^N_+}\left(|\nabla \phi |^2-\frac{\mu}{|x|^2}|\nabla \phi |^2\right) dx+ \lambda\int_{\mathbb{R}^N_+}|\phi |^2 dx -\int_{\mathbb{R}^N_+}\left((p-1)|u|^{p-2}+(q-1)|u|^{q-2}\right)\phi^2 dx,
\end{equation*}
where $\lambda\in \mathbb{R}$.
The Morse index of $u$, denote by $m(u)$, is the maximum dimension of a subspace $W\subset H_{r,\mu}$ such that $Q(\phi;u;D)<0$ for all $\phi \in W\backslash \{0\}$.

We now recall a general setting from \cite{BCJS2022TAMS} (also referenced in \cite{BL-1983}).
Let $(E, \langle\cdot,\cdot\rangle)$ and $(H,(\cdot, \cdot))$
be two infinite-dimensional Hilbert spaces that are related through a chain of continuous injections:
$E\hookrightarrow H \hookrightarrow E'$,
where the continuous injection from $E$ to $H$ has a norm bounded by 1. For simplicity, We identify $E$ with its image in $H$. 

For any $u\in E$, we define two norms: $\|u \|^2= \langle u,u \rangle$ denotes the norm in $E$, and $|u|^2= (u,u)$ denotes the norm in $H$. For any $a\in (0, +\infty)$, we define the set $S_a= \{u\in E: |u|^2= a\}$. Furthermore, we denote by $\|\cdot\|_*$ and $\|\cdot\|_{**}$ the operator norms on the spaces $\mathcal{L}(E, R)$ and $\mathcal{L}(E, \mathcal{L}(E, \mathbb{R}))$ respectively.
\begin{definition} \cite{BCJS2022TAMS}
	Let $\phi: E\to \mathbb{R}$ be a $C^2$-functional on $E$ and $\alpha\in (0,1]$. We say that $\phi'$ and $\phi''$ are $\alpha$-H\"{o}lder continuous on bounded sets, if for any $R>0$, one can find $M=M(R)>0$ such that for any $u_1, u_2\in B(0, R)$:
	\begin{equation*}\label{qingzhu}
		\|\phi'(u_1)-\phi'(u_2)\|_* \leq M\|u_1-u_2\|^\alpha,~~
		\|\phi''(u_1)-\phi''(u_2)\|_{**} \leq M\|u_1-u_2\|^\alpha.
	\end{equation*}
\end{definition}
\begin{definition} \cite{BCJS2022TAMS}
	Let $\phi$ be a $C^2$-functional on $E$. For any $u\in E$, we define the continuous bilinear map:
	\[D^2\phi(u)=\phi''(u)-\frac{\phi'(u)\cdot u}{|u|^2}(\cdot,	\cdot). \]
\end{definition}

\begin{definition}\cite{BCJS2022TAMS}
For any $u\in S_a$ and $\theta>0$, we define the approximate Morse index as
\begin{align*}
\tilde{m}_\theta(u)=\sup \{dim ~L | &L~ {\rm is~ a~ subspace ~of~ } T_uS_a~ {\rm such~ that~}\\
&D^2\phi(u)[\phi,\phi]<-\theta\|\phi\|^2, ~\forall \varphi \in L \backslash \{0\}\}.
\end{align*}
\end{definition}
If $u$ is a critical point for the constrained functional $\phi|_{S_a}$ and $\theta=0$, then $\tilde{m}_\theta(u)$ is the Morse index of $u$ as a constrained critical point.

\begin{theorem}\label{abstract theorem}
	(\cite{BCJS2022TAMS}).
	Let $I\subset(0, +\infty)$ be an interval and consider a family of $C^2$ functionals $\Phi_\rho : E \to \mathbb{R}$ of the form:
	\begin{equation*}
		\Phi_\rho(u)= A(u)-\rho B(u), ~~\rho \in I,
	\end{equation*}
where $B(u)\geq 0$ for every $u\in E$, and
\begin{equation*}
	{\rm either~~ }A(u)\to +\infty ~~{\rm or}~~ B(u)\to +\infty
	{\rm ~~as~~} u\in E ~~{\rm and ~~}\|u\|\to +\infty.
\end{equation*}	
Suppoes moreover that $\Phi'_\rho$ and $\Phi''_\rho$ are $\alpha$-H\"{o}lder continuous on bounded sets for some $\alpha \in (0,1]$.
Finally, suppose that there exist $w_1, w_2\in S_a$ (independent of $\rho$) such that, setting
\begin{equation*}
	\Gamma=\{\gamma\in C([0,1], S_a): \gamma(0)=w_1, \gamma(1)=w_2\},
\end{equation*}
we have
\begin{equation*}
	c_\rho:=\inf_{\gamma\in \Gamma}\max_{t\in [0,1]}\Phi_\rho(\gamma(t))
	>\max\{\Phi_\rho(w_1), \Phi_\rho(w_2)\},~~\forall \rho \in I.
\end{equation*}
Then, for almost every $\rho\in I$, there exist	sequences $\{u_n\}\subset S_a$ and $\zeta_n\to 0^+$ such that, as $n \to \infty$,
\begin{itemize}
	\item[(i)]  $\Phi_\rho(u_n)\to c_\rho$;
	\item[(ii)] $\|\Phi'_\rho|_{S_a}(u_n)\|\to 0$;
	\item[(iii)]  $\{u_n\}$ is bounded in $E$;
	\item[(iv)] $\tilde{m}_{\zeta_n}(u_n)\leq 1$.
\end{itemize}
\end{theorem}

\begin{lemma}\label{goo}\cite{BCJS2023Nolinearity}
	Let $\{u_n\}\subset S_a$, $\{\lambda_n\}\subset \mathbb{R}^+$ with $\zeta_n\to 0^+$. Assume that the following conditions hold:
	\begin{itemize}
		\item [(i)] For sufficiently large $n\in \mathbb{N}$, any subspace $W_n\subset E$ such that
		\begin{equation*}
			\Phi''_\rho (u_n)[\varphi,\varphi]+\lambda_n |\varphi|^2< -\zeta_n\|\varphi\|^2~~\mbox{for all}~~\varphi\in W_n\backslash \{0\},
		\end{equation*}
		satisfies: $dim W_n\leq 2$.
		\item [(ii)] There exist $\lambda\in \mathbb{R}$, a subspace $Y$ of $E$ with $dim Y\geq 3$ and $\omega>0$ such that, for sufficiently large $n\in \mathbb{N}$, \begin{equation*}\label{baobao}
			\Phi''_\rho (u_n)[\varphi,\varphi]+\lambda_n |\varphi|^2<\omega\|\varphi\|^2~~\mbox{for all}~~\varphi\in Y.
		\end{equation*}
	\end{itemize}
	Then $\lambda_n>\lambda$ for all sufficiently large $n\in \mathbb{N}$. In particular, if \eqref{baobao} holds for any $\lambda<0$, then $\liminf_{n\to \infty}\lambda_n\geq 0$.
\end{lemma}

In what follows, we obtain the main result of this section.

\begin{theorem}\label{Thm-perturb}
	Assume that $a\in (0, a_*)$. There exists $\mu_2\in (0, \frac{\bar{\mu}}{4})$ such that for any $\mu\in (0, \mu_2)$, for almost every $\rho\in [\frac{1}{2}, 1]$, there exists a critical point $u_\rho$ of $J_\rho$ on $\mathcal{S}_{r,a}$ at level $c_\rho$, which solves
\begin{equation*}\label{approximate}
\begin{cases}
-\Delta u_\rho-\frac{\mu}{|x|^2}u_\rho + \lambda_\rho u_\rho=\rho|u_\rho|^{q-2}u_\rho+\rho|u_\rho|^{p-2}u_\rho  & {\rm in} \, \mathbb{R}^N_+,\\
\frac{\partial u_\rho}{\partial \nu}=0  \, &{\rm on}\,\partial \mathbb{R}^N_+
\end{cases}
\end{equation*}
for some $\lambda_\rho>0$.
\end{theorem}

\begin{proof}
We apply Theorem \ref{abstract theorem}  to the family of functionals $J_\rho$, with $E=H_{r,\mu}$, $H=L^2(\mathbb{R}_+^N)$, $S_a=\mathcal{S}_{r,a}$.

Define the operators
\[A(u)=\frac{1}{2}\int_{\mathbb{R}_+^N}\left(|\nabla u|^2-\frac{\mu}{|x|^2}u^2 \right)dx ~~\mbox{and}~~B(u)=\frac{1}{q}\int_{\mathbb{R}_+^N}|u|^q dx+\frac{1}{p}\int_{\mathbb{R}_+^N}|u|^p dx.\]
By Lemma \ref{hardy}, we deduce that, for sufficiently small $\mu>0$, 
\[u\in \mathcal{S}_{r,a},~~\|u\|\to \infty \Rightarrow A(u)\to \infty. \]
Let $J_\rho'$ and $J_\rho''$ denote the first and second free derivatives of $J_\rho$, respectively. Clearly, both $J_\rho'$ and $J_\rho''$ are of class $C^1$, and are therefore locally H\"older continuous on $\mathcal{S}_{r,a}$. Then by Theorem \ref{abstract theorem}, for almost every $\rho \in [\frac{1}{2}, 1]$, there exist a bounded sequence $\{u_{\rho_n}\}\subset \mathcal{S}_{r,a}$, which we will still denote by $\{u_n\}$, and a sequence $\{\zeta_n\}\subset \mathbb{R}^+$ with $\zeta_n\to 0^+$, such that 
\begin{equation}\label{jinlao}
	\langle J'_\rho(u_n) +\lambda_n u_n, \varphi \rangle \to 0,~~\forall \varphi\in H_{r,\mu},
\end{equation}
where 
\begin{equation}\label{anmo}
	\lambda_n = -\frac{1}{a}\langle J'_\rho(u_n), u_n\rangle,
\end{equation}
and $\tilde{m}_{\zeta_n}(u_n)\leq 1$. By \eqref{anmo} we get that $\{\lambda_n\}$ is bounded.
Then, passing to a subsequence, there exists $\lambda_\rho\in \mathbb{R}$ such that $\lim\limits_{n\to \infty}\lambda_n=\lambda_\rho$.

In addition, since the map $u\longmapsto |u|$ is continuous and $J(u_n)=J(|u_n|)$, it is possible to choose $\{u_n\}$ such that $u_n\geq 0$ on $\mathbb{R}_+^N$. Furthermore, there exists $u_\rho\in H_{r,\mu}$ with $u_\rho \geq 0$ such that 
\begin{align*}
	&u_n\rightharpoonup u_\rho ~~\mbox{in}~~H_{r,\mu},\\
	&u_n \to u_\rho ~~\mbox{in}~~L^r(\mathbb{R}_+^N), r\in (2, 2^*),\\
	&u_n(x)\to u_\rho(x) ~~\mbox{for}~~\mbox{a.e.}~ x\in \mathbb{R}_+^N.
\end{align*}
Using \eqref{jinlao} 
we see that $u_\rho$ satisfies 
\begin{equation}\label{limit}
	-\Delta u_\rho +\lambda_\rho u_\rho-\frac{\mu}{|x|^2}u_\rho=\rho|u_\rho|^{q-2}u_\rho+\rho|u_\rho|^{p-2}u_\rho.
\end{equation}

We claim that $u_\rho \not\equiv 0$. 
To achieve this goal, we shall first prove that $\lambda_\rho \geq 0$.
Since the codimension of $T_{u_n}\mathcal{S}_{r,a}$ is $1$, we infer that if the following inequality 
\begin{equation}\label{quan}
J''_\rho(u_n)[\varphi,\varphi] +\lambda_n\int_{\mathbb{R}_+^N}\varphi^2 dx<-\zeta_n\int_{\mathbb{R}_+^N}\left(|\nabla \varphi|^2+ \varphi^2 \right)dx
\end{equation}
holds for every $\varphi\in V_n\backslash \{0\}$, where $V_n$ is a subspace of $H_{r,\mu}$, then the dimension of $V_n$ is at most $2$.
On the other hand, by Lemma \ref{thank}, we know that for every $\lambda<0$, there exists a subspace $Y$ of $H_{r,\mu}$ with $dim Y\geq 3$ and $\omega>0$ such that, for large $n\in \mathbb{N}$,
\[J''_\rho(u_n)[\varphi,\varphi]+\lambda \int_{\mathbb{R}_+^N}\varphi^2 dx\leq \int_{\mathbb{R}_+^N}\left(|\nabla\varphi|^2+\lambda \varphi^2\right) dx\leq -\omega\|\varphi\|^2,~~\forall \varphi\in Y\backslash\{0\} .\]
Therefore, by Lemma \ref{goo} we can conclude that $\lambda_\rho \geq 0$.  Using (\ref{jinlao}) we have 
\begin{eqnarray*}
\lambda_n a=(1-\frac{2}{q})\|u_n\|_q^q+(1-\frac{2}{p})\|u_n\|_p^p-2c_{\rho}+o_n(1),
\end{eqnarray*}
which together with the fact $J(u_n)\ge \frac{a}{2}$ 
 implies that 
\begin{eqnarray*}
\liminf\limits_{n\to+\infty}[(1-\frac{2}{q})\|u_n\|_q^q+(1-\frac{2}{p})\|u_n\|_p^p]\ge 2c_{\rho}\ge a.
\end{eqnarray*}
Hence $u_\rho \not\equiv 0$. By (\ref{limit}) and Lemma \ref{poho}, we have the following identity
\begin{equation*}
(\frac{1}{2}-\frac{N-2}{2N})\lambda_\rho a=(\frac{1}{q}-\frac{N-2}{2N})\|u_{\rho}\|_q^q+(\frac{1}{p}-\frac{N-2}{2N})\|u_{\rho}\|_p^p,
\end{equation*}
which implies that there exists $c_0$ independent of $\mu$ such that $\lambda_\rho>c_0>0$.

Now from \eqref{jinlao} and \eqref{limit}, we deduce that 
\begin{equation*}
		\int_{\mathbb{R}_+^N}\left[|\nabla (u_n-u_\rho)|^2 -\frac{\mu}{|x|^2}(u_n-u_\rho)^2 +\lambda_\rho (u_n-u_\rho)^2\right] dx=o_n(1).
\end{equation*}
Using Lemma \ref{hardy}, there exists $\mu_2\in (0, \frac{\bar\mu}{4})$ such that, for $\mu\in (0, \mu_2)$, we have
\begin{align*}
	&(1-\frac{2\mu}{\bar{\mu}})\int_{\mathbb{R}_+^N}|\nabla (u_n-u_\rho)|^2 dx+ \frac{\lambda_\rho}{2}\int_{\mathbb{R}_+^N}(u_n-u_\rho)^2 dx\\
	&\leq	
\int_{\mathbb{R}_+^N}\left[|\nabla (u_n-u_\rho)|^2 -\frac{\mu}{|x|^2}(u_n-u_\rho)^2 +\lambda_\rho (u_n-u_\rho)^2\right] dx=o_n(1). 
\end{align*}
This implies $u_n \to u_\rho$ strongly in $H_{r,\mu}$.
\end{proof}

\subsection{Original problem}
In this subsection,
we will demonstrate that the sequence $\{u_{\rho_n}\}$ converges to a constrained critical point of $J_1$ as $\rho_n\to 1^-$.
For simplicity, we will use the following notations: $u_n:= u_{\rho_n}$, $\lambda_n:=\lambda_{\rho_n}$, $c_n:= c_{\rho_n}$. Here $u_n$ weakly solves the following problem
\begin{equation}\label{withn}
\begin{cases}
-\Delta u_n-\frac{\mu}{|x|^2}u_n^2 + \lambda_n u_n =\rho_n |u_n|^{q-2}u_n+ \rho_n|u_n|^{p-2}u_n  & {\rm in} \, \mathbb{R}_+^N,\\
\frac{\partial u_n}{\partial \nu}=0 \, &{\rm on}\,\partial \mathbb{R}_+^N,\\
\int_{\mathbb{R}_+^N}u_n^2 dx=a,
\end{cases}
\end{equation}
where $\lambda_n\in \mathbb{R}$, $\lambda_n >0$ and $\rho_n \to 1^-$.

\begin{proof}[Completion of proof of Theorem \ref{mainhalf}]
We first show that $\{u_{n}\}$ is bounded in $H_{\mu,r}$.
By Lemma \ref{MP} and the monotonicity of $c_\rho$, we know that
\[J_1(w_1)\leq J_\rho (w_1)\leq c_\rho\leq c_{\frac{1}{2}}, ~~\forall \rho\in \left[\frac{1}{2}, 1\right], \]
which implies that
 $\{c_n\}$ is bounded.
From \eqref{withn} and Lemma \ref{poho}, 
it follows that
\begin{equation}\label{uzuh}
	\frac{\rho_n}{p}\int_{\mathbb{R}_+^N}|u_n|^p dx=\frac{2}{N(p-2)}\int_{\mathbb{R}_+^N}\left( |\nabla u_n|^2-\frac{\mu}{|x|^2}u_n^2\right) dx -\frac{\rho_n(q-2)}{q(p-2)}\int_{\mathbb{R}_+^N}|u_n|^q dx.
\end{equation}
Then, using the Gagliardo-Nirenberg inequality and Lemma \ref{hardy} with $\delta=\frac{1}{\bar\mu}$,
\begin{align*}
	c_n &=\left(\frac{1}{2}-\frac{2}{N(p-2)}\right)\int_{\mathbb{R}_+^N}\left( |\nabla u_n|^2-\frac{\mu}{|x|^2}u_n^2\right) dx -\rho_n\left(\frac{1}{q}-\frac{q-2}{q(p-2)}\right)\int_{\mathbb{R}_+^N}|u_n|^q dx\\
	&\geq \left(\frac{1}{2}-\frac{2}{N(p-2)}\right)\left[(1-\frac{2\mu}{\bar{\mu}})\int_{\mathbb{R}_+^N}|\nabla u_n|^2 dx - \mu C_{\bar{\mu}}a \right]-\left(\frac{1}{q}-\frac{q-2}{q(p-2)}\right)C_{N,q}^qa^{q(1-\gamma_q)}\|u_n\|^{q\gamma_q}.
\end{align*}
Since $q\in (2, 2+\frac{4}{N})$, it follows that 
$\{u_n\}$ is bounded and thus $\{\lambda_n\}$ is also bounded.
Then passing to a subsequence, there exists $\lambda_*\in \mathbb{R}$ such that $\lim\limits_{n\to \infty}\lambda_n=\lambda_*$.
Furthermore, there exists $u_*\in H_{\mu,r}$ with $u_* \geq 0$ such that 
\begin{align*}
	&u_n\rightharpoonup u_* ~~\mbox{in}~~H_{\mu,r},\\
	&u_n \to u_* ~~\mbox{in}~~L^r(\mathbb{R}_+^N), r\in (2, 2^*),\\
	&u_n(x)\to u_*(x) ~~\mbox{for}~~\mbox{a. e.}~ x\in \mathbb{R}_+^N.
\end{align*}
By similar arguments to those in the proof of Theorem \ref{Thm-perturb}, we can deduce that $\lambda_*\ge c_1>0$ independent of $\mu$. 
Denote $\mu_*=\min\{\mu_1, \mu_2\}$. Then
for $\mu\in (0, \mu_*)$,  $u_n\to u_*$ strongly in $H_{r,\mu}$, and that $u_*>0$. This completes the proof.
\end{proof}

\section*{Acknowledgements}
The research of Xiaojun Chang was partially supported by National Natural Science Foundation of China (Grant No.12471102) and the
Research Project of the Education Department of Jilin Province (JJKH20250296KJ). 

\section*{References}
\biboptions{sort&compress}

\end{document}